\documentclass[letterpaper, 11pt,  reqno]{amsart}
\usepackage[margin=1.2in,marginparwidth=1.5cm, marginparsep=0.5cm]{geometry}
\usepackage{amsmath,amssymb,amsthm,latexsym}
\usepackage[abbrev]{amsrefs}
\usepackage{mathrsfs}

\newtheorem{theorem}{Theorem}[section]
\newtheorem{lemma}[theorem]{Lemma}

\theoremstyle{definition}

\newtheorem{corollary}{Corollary}[section]
\newtheorem{proposition}{Proposition}[section]

\theoremstyle{remark}
\newtheorem{remark}[theorem]{Remark}

\numberwithin{equation}{section}

\newcommand{\bbr}{\mathbb{R}}

\newcommand{\ve}{\varepsilon}

\newcommand{\N}{{\mathbb N}}

\newcommand{\R}{{\mathbb R}}

\newcommand{\eps}{{\varepsilon}}

\begin{document}

\title[Brezis Browder results and $s$-harmonic fuctions]{Elementary Brezis--Browder type results and Representation formulae for $s$--harmonic functions}

\author{Damiano Greco}
\address{Damiano Greco\\School of Mathematics\\
The University of Edinburgh\\
and The Maxwell Institute for the Mathematical Sciences\\
James Clerk Maxwell Building\\
The King's Buildings\\
Peter Guthrie Tait Road\\
Edinburgh\\ 
EH9 3FD\\
 United Kingdom}
\email{dgreco@ed.ac.uk}


\date{}

\begin{abstract}
\noindent We prove Brezis--Browder type results for fractional Sobolev spaces and quantitative type estimates for $s$--harmonic functions.
Furthermore, we give sufficient conditions 
for distributional solutions to the fractional Poisson's equation $(-\Delta)^su=T$ on $\R^d$ to be of the form
$$u(x)=\int_{\R^d}\frac{T(y)}{|x-y|^{d-2s}}dy+l,\quad l\in \R.$$
\end{abstract}

\maketitle

\section{Introduction}
In this paper we first focus on the study of $s$-harmonic functions on $d$-dimensional balls, namely on {distributional solutions} to the equation\footnote{See eq. \eqref{11_lap} for the notion of distributional solution.} $(-\Delta)^{s}u=0$ in $B(0,R)$, where 
 $R>0$. Historically, in the case of the standard Laplace operator $-\Delta$  (corresponding to $s=1$), it is well known \cite{hypo} that any distributional solution to $-\Delta u=0$ in $B(0,R)$ is smooth; see also \cite{garofalo}*{Section 12} and reference therein for broader introduction to the topic. 
 The same result holds for general $s$, as well as for a wider class of pseudodifferential operators; see \cite{garofalo}*{Theorem 12.19}. In Section \ref{sec3} of the present paper, we carry out such analysis and obtain quantitative estimates for $s$--harmonic functions on balls.  For $0<s<1$ similar outcomes were proved in \cites{BBK2,entire}. By following an argument developed in \cite{orbital}, we  extend some of them to the range $d\ge 1$ and  $0<s<\tfrac{d}{2}$. 
We refer the reader to Lemma \ref{lemma_1} for the precise statement. These estimates, apart from their intrinsic interest, admit various applications. 
The first one is that $s$--harmonic functions are polynomials. Note that, for a generic $s>0$, control on the behavior of a function $u$ at infinity is required in order to define $(-\Delta)^{s}u$ in the distributional sense. This is due to the \textit{non-local} nature of $(-\Delta)^s$, whose action on smooth functions with compact support has not (in general) compact support but only polynomial decay, see e.g., \cite{mazya}*{Lemma 1.2} or \cite{Abatangelo_1}*{Lemma B.5}. For this reason, it is classical to introduce the space $\mathscr{L}^{1}_{2s}(\R^d)$ defined by 
\begin{equation}\label{L_pesato}
\mathscr{L}^{1}_{2s}(\bbr^d):=\left\{u\in L^1_{loc}(\bbr^d):\int_{\bbr^d}\frac{|u(x)|}{(1+|x|)^{d+2s}}dx<\infty\right\},
	\end{equation}
see also \cite{BBK}*{Section 3}.
As a result, we formulate the following:
\begin{theorem}\label{poli}
	Let $d\ge 1$ and $0<s<\frac{d}{2}$. If  $u\in \mathscr{L}^{1}_{2s}(\R^d)$ solves 
	\begin{equation*}\label{eq_001}
		(-\Delta)^{s}u=0\quad \text{in}\ \mathscr{D}'(\R^d),
	\end{equation*}
	then $u$ is a polynomial of degree strictly smaller than $2s$.
\end{theorem}
\noindent
Note that, for $d\ge 2$,  Theorem \ref{poli} extends the validity of \cite{entire}*{Theorem 1.1}  from $s\in (0,1)$ to  $s\in (0,\tfrac{d}{2})$.

 Next, we focus on the (fractional) Poisson's equation 
\begin{align}
(-\Delta)^{s}u=f\quad  \text{in}\ \R^d,
\label{frac_intro}
\end{align}
under suitable assumptions on $f$. For example, if $f\in L^{p}(\R^d)$ with $1<p<\tfrac{d}{2s}$, it is well known that a solution to the Poisson's equation is given by the \textit{Riesz potential} of $f$: 
 \begin{align}\label{RZ}
 (I_{2s}*f)(x)=A_{2s}\int_{\R^d}\frac{f(y)}{|x-y|^{d-2s}}dy, 
 \end{align}
where by $I_{\alpha}(x):=A_{\alpha}|x|^{\alpha-d}$ we denote the Riesz kernel with $\alpha\in (0,d)$, and $*$ stands for the standard convolution in $\R^d$. The choice of the normalization constant $A_{\alpha}{:=\!\frac{\Gamma((d-\alpha)/2)}{\pi^{d/2}2^{\alpha}\Gamma(\alpha/2)}}$
ensures the validity of the semigroup property $I_{\alpha+\beta}=I_\alpha*I_\beta$ for all $\alpha,\beta\in(0,d)$ such that $\alpha+\beta<d$, see \cite{DuPlessis}*{pp.\thinspace{}73--74}. 
 In view of Theorem \ref{poli},  we derive that  $I_{2s}*f$ in \eqref{RZ} is the unique (within a suitable class of functions) distributional solution  of \eqref{frac_intro}, see e.g.,  Remark \ref{RZ_ext}.
This result can be seen as the higher order version of \cite{entire}*{Corollary 1.4}.
Furthermore, in Theorem \ref{lemmamit} we deduce a representation formula for distributional solutions to the fractional Poisson's equation in the spirit of \cite{mitidieri}*{Theorem 2.4} where a similar result has been obtained in the case of the polyharmonic operator $(-\Delta)^{k}$, $k$ being an integer number strictly larger than one. See also Remark \ref{remarknew} for similar outcomes.

 In order to clearly state Theorem \ref{lemmamit} below, we first recall that by $\dot H^{s}(\R^d)$ we denote the homogeneous $L^{2}$-based Sobolev space (see \eqref{norm_intro} with $p=2$) while $\dot H^{-s}(\R^d)$ denotes its dual.
\begin{theorem}\label{lemmamit}
	Let $d\ge 1$,  $0<s<\tfrac{d}{2}$, $0 \le T\in \dot{H}^{-s}(\R^d)\cap L^{1}_{loc}(\R^d)$ and $l\in \R$. The following are equivalent:
	\begin{itemize}
		\item[$(i)$] $u\in \mathscr{L}^{1}_{2s}(\R^d)$ satisfies 
		\begin{equation*}\label{eqT}
			(-\Delta)^su=T\quad \mathscr{D}'(\R^d)
		\end{equation*}
		and
		\begin{equation}\label{eq2228}
			\liminf_{R\to \infty}\frac{1}{R^d}\int_{R<|x-y|<2R} |u(y)-l |dy<\infty\quad for\ a.e.\ x\in \R^d;
		\end{equation} 
		\item[$(ii)$] $u$ can be written as
		$$u(x)=A_{2s}\int_{\R^d}\frac{T(y)}{|x-y|^{d-2s}}dy+l\quad for\ a.e.\ x\in \R^d.$$
	\end{itemize}
\end{theorem}

In the second part of this paper, we move our attention to Brezis--Browder type results for fractional Sobolev spaces. If  $d$ and $m$ are integers larger than one, $p$ is a real number strictly larger than one, $p':=\frac{p}{p-1}$ represents the conjugate of $p$, and $W^{m,p}(\R^d)$ denotes the standard Sobolev space,  H. Brezis and F. Browder in \cite{1982} proved that if $T\in W^{-m,p'}(\R^d)\cap L^1_{loc}(\R^d)$ and $u\in W^{m,p}(\R^d)$ are such that $T(x)u(x)\ge -|f(x)|$ a.e. for some integrable function $f$ then $Tu\in L^1(\R^d)$ and 
\begin{align}
\label{inter_B}
\langle T,u \rangle_{W^{-m,p'}(\R^d),\, W^{m,p}(\R^d)}=\int_{\R^d}T(x)u(x)dx.
\end{align}
The first order case $m=1$ was already studied in  \cite{cartan} where the same authors provide also counterexamples; see also \cite{1984} for similar outcomes.
The proof of the above results heavily relies on a truncation procedure introduced by L. Hedberg (see \cite{Adams}*{Theorem 3.4.1} and \cite{Hedbger})  stating that for every $u\in W^{m,p}(\R^d)$ there exists a sequence $(u_n)_n$  such that 
$$ \begin{cases} u_n\in W^{m,p}(\R^d)\cap L^{\infty}(\R^d), & \,\,\text{supp}(u_n)\, \mbox{is compact}; \\ |u_n(x)|\le |u(x)|\ \mbox{and}\ u_n(x)u(x)\ge 0 & \mbox{ a.e. in}\ \R^d; \\
	u_n\to u\ \mbox{in}\ W^{m,p}(\R^d).& 
\end{cases}
$$

\noindent
Such result has been extended by considering $\Omega$ a  generic open subset of $\R^d$, $T=~{\mu+h}$ for some Radon measure $\mu$ and  $h\in L^1_{loc}(\Omega)$. Namely, in \cite{1990} the authors proved that if $u(x)\ge 0$ and 
$h(x)u(x)\ge -|f(x)|$ a.e. in $\Omega$ for some $f\in L^1(\Omega)$ then $u\in L^{1}(\Omega, d\mu)$, $hu\in L^1(\Omega)$ and 
$$\langle T,u \rangle_{W^{-m,p'}(\Omega),\, W^{m,p}_0(\Omega)}=\int_{\Omega}u(x)d\mu(x)+\int_{\Omega}h(x)u(x)dx.$$
We further mention that  the case $m=1$, $h=0$ and $\Omega$ an open subset of $\R^d$ was studied in \cite{cartan} and, already  in \cite{1982},  the case $m>1$ was analyzed under some restriction on $T$, see \cite{1982}*{Theorem 3}. 

In the sequel, we will often refer to ``validity of Brezis--Browder results" when an equality of the type \eqref{inter_B} holds.  In particular,  Section \ref{sec4} focuses on the fractional counterpart of the above mentioned results. For the remaining part of the introduction, we present our outcome in the homogeneous setting only; see Theorem \ref{integr_alpha_big} for a complete statement.  

Let $d\ge 1$, $s>0$ and $1<p<\frac{d}{s}$. Let us consider the homogeneous Sobolev spaces $\dot{H}^{s,p}(\R^d)$ defined as the completion of $C^{\infty}_c(\R^d)$ with respect to the norm
\begin{align}
\|u\|_{\dot H^{s,p}(\R^d)}:=\Big|\!\Big|\mathscr{F}^{-1}\Big({|\cdot|}^{s}\mathscr{F}(u)\Big)\Big|\!\Big|_{L^p(\R^d)}, 
\label{norm_intro}
\end{align}
where $\mathscr{F}$ denotes the Fourier transform operator. 
In order to deal with  $s\notin \N$,  the strategy is to approximate any function in $\dot{H}^{s,p}(\R^d)$ by multiplication against a suitable family of smooth functions. However, such an approximating sequence is not in $L^{\infty}(\R^d)$ as in the Hedberg case. Nevertheless, it becomes elementary to overcome this issue by further requiring higher  local integrability of the element $T$ generating a linear functional on $\dot{H}^{s,p}(\R^d)$ (i.e., on $T\in \dot H^{-s,p'}(\R^d)$). 
Namely, for any $\eps>0$, if we require $T\in \dot{H}^{-s,p'}(\R^d)\cap L^{\bar{q}}_{loc}(\R^d)$ where $\bar{q}=\bar{q}(d,s,p)$ is defined by 
\begin{equation}\label{q_}
	\bar{q}(d,s,p):=\footnote{Note that, unlike the space $\dot H^{s,p}(\R^d)$, $\bar{q}$ is defined also for $p\ge \frac{d}{s}$. The reason we included such cases is because the same exponent will play a role in the inhomogeneous setting, which is in fact well defined for $p\ge \frac{d}{s}$.} \begin{cases} 1, & \mbox{if }s\in \N\ \text{or}\ p>\frac{d}{s}, \\ 1+\eps, & \mbox{if } s\notin \N\ \text{and}\ p=\frac{d}{s},\\
	\frac{dp}{d(p-1)+sp}, & \mbox{if } s\notin \N\ \text{and}\ 1<p<\frac{d}{s},
	\end{cases}
\end{equation}
we prove the following:
\begin{theorem}\label{integr_alpha_intro}
	Let  $d\ge 1$, $s>0$, $1<p<\frac{d}{s}$ and $\bar{q}$ as in \eqref{q_}. Assume that $T\in \dot{H}^{-s,p'}(\bbr^d)\cap L_{loc}^{\bar{q}}(\bbr^d)$ and $u~{\in \dot{H}^{s,p}(\bbr^d)}$ are such that $Tu\ge -|f|$  for some  $f\in L^{1}(\bbr^d)$. Then, $Tu\in L^{1}(\bbr^d)$ and 
	$$\langle T,u \rangle_{\dot{H}^{-s,p'}(\R^d), \,\dot{H}^{s,p}(\R^d)}=\int_{\bbr^d}T(x)u(x)dx.$$
\end{theorem}

\section{Preliminaries}\label{sec2}
\subsection*{Fractional Laplacian and Sobolev spaces}
\noindent Let $d\ge 1$ and $s>0$. The fractional Laplacian $(-\Delta)^{s}$  on $\mathbb{R}^d$ of a function in $C^{\infty}_c(\R^d)$ is defined by means of the Fourier transform
\begin{equation}\label{laplac_fourier}
	(-\Delta)^{s}u:=\mathscr{F}^{-1}\left({|\cdot|}^{2s}\mathscr{F}(u)\right).
\end{equation}
Thus, in view of \eqref{norm_intro} and \eqref{laplac_fourier},  the homogeneous fractional Sobolev space $\dot{H}^{s,p}(\bbr^d)$ can be defined as the completion of $C^{\infty}_c(\R^d)$ with respect to the norm
\begin{equation*}
	\|u\|_{\dot{H}^{s,p}(\bbr^d)}:=\|(-\Delta)^{\frac
	{s}{2}}u\|_{L^p(\R^d)},
\end{equation*}
see e.g., \cite{mult_mazya}*{Chapter 3, page 70}. If $1<p<\frac{d}{s}$, then $\dot H^{s,p}(\R^d)$ can be identified with a space of functions having $C^{\infty}_c(\R^d)$ as a dense subspace, i.e., it is isomorphic to the \textit{Riesz potential space} often denoted by $I_{s}(L^p)$, see \cite{Samko}*{Chapter 5, Corollary 1} and  \cite{mazya}*{Lemma 1.8}. In particular, 
\begin{align}\label{hom_embedding}
\dot H^{s,p}(\R^d)\hookrightarrow L^{p^*}(\R^d),\quad  p^*:=\frac{dp}{d-sp}.
\end{align}
 On the other hand, it is well known that for $p\ge  \frac
{d}{s}$ the homogeneous fractional Sobolev spaces can not be identified as spaces of tempered distributions; cf. \cite{realization}*{Theorem 4}.
Similarly, if $1<p<\infty$ and $s>0$ we can define the inhomogeneous fractional Sobolev space as the completion of $C^{\infty}_c(\R^d)$ with respect to the norm
\begin{align*}
\|u\|_{H^{s,p}(\R^d)}:=\|{(1-\Delta)}^{\frac{s}{2}}u\|_{L^p(\R^d)},\quad {(1-\Delta)}^{\frac{s}{2}}u:=\mathscr{F}^{-1}\left({(1+|\cdot|}^{2})^{\frac{s}{2}}\mathscr{F}(u)\right).
\end{align*}
Such a space is equivalent to the \textit{Bessel potential space} often denoted by $G_s(L^p)$; see e.g.,  \cite{Samko}*{eq. (27.27)}. In particular, if $s\in \N$ then $H^{s,p}(\R^d)$ coincides with $W^{s,p}(\R^d)$, see \cite{Adams}*{Theorem 1.2.3}. Moreover, the following embeddings hold (\cite{mult_mazya}*{Theorem 3.1.3}):
\begin{align}\label{embeddings_in}
H^{s,p}(\R^d)\hookrightarrow \begin{cases} L^{q}(\R^d), & \mbox{if }1<p<\frac{d}{s}\ \text{and}\  p\le q\le p^*,\ \mbox{or }p=\frac{d}{s}\ \text{and}\  p\le q<\infty,  \\ L^{\infty}(\R^d), & \mbox{if }p>\frac{d}{s}.
\end{cases}
\end{align}
In what follows, we denote by $\dot H^{-s,p'}(\R^d)$ (respectively $H^{-s,p'}(\R^d)$) the dual space of $\dot H^{s,p}(\R^d)$ (respectively $H^{s,p}(\R^d)$) and $\langle \cdot, \cdot \rangle_{\dot{H}^{-s,p'}\!,\dot{H}^{s,p}}$ (respectively $\langle \cdot, \cdot \rangle_{{H}^{-s,p'}\!,{H}^{s,p}}$) the corresponding duality pairing. In the Hilbert setting ($p=2$), we for the sake of the reader, we drop the exponent $p$ (respectively $p'$) in the definition of $\dot H^{s,p}(\R^d)$ (respectively its dual).
 Note that, if $0<s<\frac{d}{2}$, Riesz representation theorem guarantees that for every $T\in \dot{H}^{-s}(\bbr^d)$  there exists a unique element ${U}_{T}\in \dot{H}^{s}(\bbr^d)$ such that 
\begin{equation}\label{weak_laplac}
	\langle T,\varphi \rangle_{\dot{H}^{-s}, \,\dot{H}^s}=\langle {U}_T, \varphi \rangle_{\dot{H}^{s}(\bbr^d)}\quad  \forall \varphi\in \dot{H}^{s}(\bbr^d).
\end{equation}
 Moreover, 
\begin{equation}\label{nomr_4}
	\left\|{U}_T\right\|_{\dot{H}^{s}(\bbr^d)}=\left\|T\right\|_{\dot{H}^{-s}(\bbr^d)},
\end{equation}
so that the duality \eqref{weak_laplac} is an isometry. 
The function ${U}_T$ satisfying \eqref{weak_laplac} is called a weak solution of the equation 
\begin{equation}\label{weak_10}
	(-\Delta)^{s}u=T \quad \text{in}\  \dot{H}^{-s}(\bbr^d).
\end{equation}
	\indent Assume now that $u\in C^{\infty}_c(\R^d)$. Then, for every $m\in \N$ and $s\in (0,m)$ one can define the operator	
	\begin{equation}\label{L}
		L_{m,s}u(x):=\frac{C_{d,m,s}}{2}\int_{\R^d}\frac{\delta_{m}u(x,y)}{|y|^{d+2s}}dy, \quad x\in \R^d,
	\end{equation}
	where $\delta_m u(x,y)$ is defined by 
	\begin{equation}\label{sigma}
		\delta_m u(x,y):=\sum_{k=-m}^{m}(-1)^{k}\binom{2m}{m-k}u(x+ky)\quad \forall x,y\in \R^d,
	\end{equation}
	and   $C_{d,m,s}$ is a suitable positive constant, cf. \cite{Abatangelo_2}*{eq. (1.2)}.
In particular, it is known that
	$L_{m,s}$ has $|\xi|^{2s}$ as a Fourier multiplier. Namely, for every $m\in \N$ and $s\in (0,m)$ it holds
	\begin{equation}\label{fourier_L}
		L_{m,s}\varphi=\mathscr{F}^{-1}\left({|\cdot|}^{2s}\mathscr{F}(\varphi)\right)=	(-\Delta)^{s}\varphi\quad  \forall \varphi\in C^{\infty}_c(\R^d),
	\end{equation}
	see \cite{Abatangelo_2}*{Theorem 1.9}.
	In particular, in view of \eqref{fourier_L},  if $u$ belongs to the space $\mathscr{L}^{1}_{2s}(\bbr^d)$ defined by \eqref{L_pesato}
	then $(-\Delta)^{s}u$ defines a distribution on every open set $\Omega\subset \R^d$ by 
	\begin{equation}\label{11_lap}
		\langle (-\Delta)^{s}u, \varphi \rangle :=\int_{\R^d}u(x)L_{m,s}\varphi(x) dx=\int_{\R^d}u(x)(-\Delta)^s\varphi(x)dx\quad \forall \varphi\in C^{\infty}_c(\Omega).
	\end{equation}
Thus, if $u\in \mathscr{L}^1_{2s}(\R^d)$ and $T$ is a distribution on $\Omega$, we write 
$$(-\Delta)^{s}u=T\quad \text{in}\ \mathscr{D}'(\Omega)$$
if 
$$\langle (-\Delta)^{s}u, \varphi \rangle= \langle T,\varphi\rangle\quad \forall \varphi\in C^{\infty}_{c}(\Omega),$$
where $\langle (-\Delta)^{s}u, \varphi \rangle$ is defined by \eqref{11_lap}.
	Moreover,   if $\Omega$ is an open subset of $\R^d$ and $u\in C^{2s+\ve}(\Omega)\cap \mathscr{L}^1_{2s}(\R^d)$ then $L_{m,s} u$ as well as many  well defined pointwisely in $\Omega$ and 
	\begin{equation*}
		\int_{\R^d} L_{m,s}u(x)\varphi(x)dx=\int_{\R^d}u(x)L_{m,s}\varphi(x)dx\quad \forall \varphi\in C^{\infty}_c(\Omega),
	\end{equation*}
see \cite{Abatangelo_2}*{Lemma 1.5} for the details.
If $0<s<1$, it is worth mentioning that the fractional Laplacian $(-\Delta)^s$  arises naturally as the infinitesimal generator of the standard rotation invariant,  $2s$-stable, $\R^d$-valued L\'evy process. See in particular \cite{BBK}*{Section 3} and \cite{BBK2}*{Section 2}.

 Finally, we recall that if $T\in \dot{H}^{-s}(\R^d)$ then the function $U_T$ defined by \eqref{weak_laplac} solves \eqref{weak_10} in the distributional sense. As a matter of fact, by the definition of the inner product in $\dot{H}^s(\R^d)$ we obtain that
	\begin{equation}\label{weak_distr}
		\langle {U}_T, \varphi \rangle_{\dot{H}^{s}(\bbr^d)}\!=\!\int_{\R^d}U_T(x)(-\Delta)^s \varphi(x)dx\quad \forall \varphi\in C^{\infty}_c(\R^d),
	\end{equation}
	where we used that $U_T\in \mathscr{L}^1_{2s}(\R^d)$ and the decay of $(-\Delta)^s\varphi$ (see e.g.,  \cite{mazya}*{Lemma 1.2} or \cite{Abatangelo_1}*{Lemma B.5}).

	\subsection*{Regular distributions}
	\noindent
	Let $d\ge 1$, $s>0$ and $1<p<\infty$. Let us define $X$ as 
	\begin{align}\label{X}
	{X}:= \dot H^{s,p}(\R^d)\ \big(\text{by further assuming}\ 1<p<\tfrac{d}{s}\big)\ \text{or}\ H^{s,p}(\R^d),
	\end{align}
and $X'$ as the dual of $X$. Then, we say that $T\in X{'}\cap L^{1}_{loc}(\R^d)$ if 
	\begin{equation}\label{rho_distribution}
		\langle T,\varphi\rangle=\int_{\bbr^d}T(x)\varphi(x)dx\quad \forall \varphi\in C^{\infty}_c(\bbr^d),
	\end{equation}
and there exists a positive constant $C$ independent of $T$ such that
	\begin{equation}\label{ext_linear}
		|\langle T,\varphi\rangle|\le C\left\|\varphi\right\|_{X} \quad \forall \varphi\in C^{\infty}_c(\bbr^d).
	\end{equation}
	From \eqref{ext_linear}, we infer that $T$ can be identified as the unique continuous extension with respect to the $X$--norm of the linear functional defined by \eqref{rho_distribution}. 
Although the above notion admits various extensions/modifications; see e.g., \cite{mazya}*{eq. \!\!(1.38)}, here, as a matter of simplicity, we gave a presentation in terms of the spaces employed in the paper.
	

\section{s-harmonic functions and Poisson's equation}\label{sec3}
We begin this section by proving interior regularity for $s$--harmonic functions on balls. This result is related to \cite{garofalo}*{Theorem 12.19}.  Nevertheless, we establish a quantitative version in the spirit of \cite{entire}*{Lemma 3.1} or \cite{orbital}*{Lemma 4.2}. 

Assume that  $0<s<1$. In \cite{entire}, the author employed a regularized version of the Poisson's  Kernel of the fractional Laplacian with respect to the ball  \cite{BBK}*{Lemma 3.11} to represent $s$-harmonic functions as a convolution with such a regularized Kernel; see \cite{entire}*{eq. (2.5)}. In Lemma \ref{lemma_1}, we develop a similar argument which is instead inspired by the one in \cite{orbital}*{Lemma 4.2}. For completeness, we further refer the reader to \cite{BBK2} for gradient estimates on $s$-harmonic functions; in that work it is also shown that the strong Liouville theorem extends to the fractional setting. See  \cite{BBK2}*{Lemma 3.2}.
	
	\begin{lemma}\label{lemma_1}
		Let $d\ge 1$,  $0<s<\frac{d}{2}$ and $R>0$. Assume that $u\in \mathscr{L}^{1}_{2s}(\R^d)$ solves 
		\begin{equation*}\label{eq_00}
			(-\Delta)^{s}u=0\quad \text{in}\ \mathscr{D}'(B_{2R}).
		\end{equation*}
		Then $u\in C^{\infty}(B_R)$. Moreover, for every multi-index $n\in \N^{d}$,
		$$\|D^{n}u\|_{L^{\infty}({B_R})}\le CR^{2s-|n|}\int_{\R^d}\frac{|u(x)|}{(R^{d+2s}+|x|^{d+2s})}dx$$
		for some positive constant $C$ independent of $R$.
	\end{lemma}
	\begin{proof}
		Let $\eta\in C^{\infty}(\R)$ such that $\eta(x)=1$ if $|x|\ge 2$, $\eta(x)=0$ if $|x|<\frac{3}{2}$ and $0\le \eta(x)\le 1$ for every $x\in \R$. Let now $\varphi\in C^{\infty}_c(\R^d)$ be supported in $B_{R}$. Then, the function  $\psi:=I_{2s}*\varphi$ satisfies (see \cite{Stein}*{Chapter V, Lemma 2})
		\begin{equation}\label{eq_1}
			(-\Delta)^{s}\psi=\varphi\quad \text{in}\ \mathscr{D}'(\R^d). 
		\end{equation}
		In particular, $\psi\in C^{\infty}(\R^d)\cap \mathscr{L}^{1}_{2s}(\R^d)$.
		Now, let us fix $m\in \N$, $m>s$. Then, from \cite{Abatangelo_2}*{Lemma 1.5} we have that 
		\begin{equation}\label{eq_2}
			\int_{\R^d}\psi(x)L_{m,s}\chi(x)dx=\int_{\R^d}L_{m,s}\psi(x)\chi(x)dx\quad \forall \chi\in C^{\infty}_c(\R^d).
		\end{equation}
		In particular, by combinining \eqref{eq_1} with \eqref{eq_2} we obtain that 
		$L_{m,s}\psi(x)=\varphi(x)$ for every $x\in \R^d$.
		Next, let us consider the function $(1-\eta_R)\psi$ where $\eta_R(x):=\eta(|x|/R)$. Note that $(1-\eta_R)\psi\in C^{\infty}_c(\R^d)$ and it is supported in $B_{2R}$. Hence, we can  test \eqref{eq_00} against $(1-\eta_R)\psi$ obtaining 
		\begin{equation}\label{eq_3}
			\begin{split}
				0\!=\!& \int_{\R^d}u(x)L_{m,s}((1-\eta_R)\psi)(x)dx \!=\!\int_{\R^d}u(x)L_{m,s}\psi(x)dx \!-\!\!\int_{\R^d}u(x)L_{m,s} (\eta_R\psi)(x)dx\\
				&{=} \int_{B_{R}}u(x)\varphi(x)dx-\int_{\R^d}u(x)L_{m,s} (\eta_R \psi)(x)dx.
			\end{split}
		\end{equation}
	
		Then, by \eqref{L}, \eqref{sigma}, \eqref{eq_3} and  Fubini's theorem we obtain that there exists a positive constant $K_{d,m,s}$ such that
		\begin{equation}\label{u_repr_2}
			\begin{split}
				&K_{d,m,s}\int_{\R^d}u(x)L_{m,s} (\eta_R \psi)(x)dx\\
				&\!=\! \int_{B_R}\int_{\R^d}\int_{\R^d} \left(\!\frac{1}{|y|^{d+2s}}\left(\sum_{k=-m}^{m}(-1)^{k}\binom{2m}{m-k}\frac{\eta_R(x+ky)}{|x+ky-z|^{d-2s}}\!\!\right)dy\, u(x)dx\! \!\right)\!\varphi(z)dz\\
				&\!=K_{d,m,s}\int_{B_R}\left(\int_{\R^d}J_{R}(x,z)u( x)dx \right)\varphi(z)dz=K_{d,m,s}\int_{B_{R}}u(z)\varphi(z)dz,
			\end{split}
		\end{equation}
		where $K_{d,m,s}=(2/(A_{2s} C_{d,m,s}))$ and
		\begin{equation}\label{J_1}
			J_{R}(x,z):=\frac{A_{2s} C_{d,m,s}}{2}\int_{\R^d}\frac{1}{|y|^{d+2s}}\left(\sum_{k=-m}^{m}(-1)^{k}\binom{2m}{m-k}\frac{\eta_R(x+ky)}{|x+ky-z|^{d-2s}}\right)dy.
		\end{equation}
		In particular, from \eqref{u_repr_2} we deduce that for almost every $z\in B_R$
		\begin{equation}\label{u_repre}
			u(z)=\int_{\R^d}J_R(x,z)u(x)dx.
		\end{equation}
		From now on, we focus on proving estimates for \eqref{u_repre}.
		First of all we notice that the function $i_R(x,z)$ defined by 
		\begin{equation*}
			i_R(x,z):=\frac{\eta_R(x)}{|x-z|^{d-2s}}
		\end{equation*} 
		satisfies the following the properties:
		\begin{itemize}
			\item if $|z|<R$ and $|x|<\frac{3R}{2}$ then $i_R(x,z)=0$;
			\item if $|z|< R$ and  $|x|\ge \frac{3R}{2}$ then $|x-z|\ge \frac{R}{2}.$
		\end{itemize}
		In particular, $i_R(x,z)\in C^{\infty}(\R^d\times {B_R})$. 
		Moreover, in view of the above analysis, the function $i_R(\cdot, z)\in C^{\infty}({B_R})\cap \mathscr{L}^{1}_{2s}(\R^d)$ and $J_R(\cdot,z)=A_{2s} L_{m,s}(i_R(\cdot,z))$ from which \eqref{J_1} is well defined.
Furthermore, by differentiating, for every $n\in \N^{d}$ multi-index the following inequalities hold
\begin{align}
& |D^{n}_z i_R(x,z)|\le C(R+|x-z|)^{-(|n|+d-2s)}, \label{ineq_1}\\
& |D^{2m}_x D^{n}_{z}i_R(x,z)|\le CR^{-2m} (R+|x-z|)^{-(|n|+d-2s)},\label{ineq_2}
\end{align}
where $C$ is a positive constant not depending on $R$. For the convenience of reader we point out that, unless stated, all the constants from now on depend only on $d,m,s$ and $n$.
		By combining \eqref{ineq_1}, \eqref{ineq_2} with  \cite{Abatangelo_2}*{Lemma 2.4} we obtain that 
		\begin{equation}\label{diff_1}
			\begin{split}
				& \int_{\R^d}  \frac{1}{|y|^{d+2s}}  \left|\sum_{k=-m}^{m}(-1)^{k}\binom{2m}{m-k}  D^{n}_z i_R(x+ky,z)\right|dy\\
				&=\int_{B_R}\frac{1}{|y|^{d+2s}}\left|\sum_{k=-m}^{m}(-1)^{k}\binom{2m}{m-k}D^{n}_z i_R(x+ky,z)\right|dy\\
				&\hphantom{XXX} +\int_{B^c_R}\frac{1}{|y|^{d+2s}}\left|\sum_{k=-m}^{m}(-1)^{k}\binom{2m}{m-k}D^{n}_z i_R(x+ky,z)\right|dy\\
				&\!\!\le \!\left\|D^{2m}_x\! D^{n}_{z}i_R(\cdot ,z)\right\|_{L^{\infty}(\R^d)}\!\!\int_{B_R}\!\!|y|^{2m-(d+2s)}dy\!\\
				&\hphantom{XXX}+\!C_m\!\left\|D^{n}_z i_R(\cdot,z)\right\|_{L^{\infty}(\R^d)}\!\!\int_{B^c_R}\!\!|y|^{-(d+2s)}dy\\
				& \le C R^{-(|n|+d)},
			\end{split}
		\end{equation}
		for some positive constant $C$.
		Thus, \eqref{diff_1} implies that $J_R(x, \cdot )\in C^{\infty}({B_R})$ and 
		\begin{equation*}
			|D^{n}_z J_R(x,z)|\le  CR^{-(|n|+d)}\quad \forall (x,z)\in \R^d\times B_R.
		\end{equation*}
		Next, we further claim that for every $n\in \N^{d}$ multi-index there exists a positive constant $C$  such that 
\begin{equation}\label{diff_2}
|D^{n}_z J_R(x,z)|\le  \frac{C
}{R^{|n|-2s}(R^{d+2s}+|x|^{d+2s})}\quad \forall (x,z)\in \R^d\times B_R.
\end{equation}
To this aim, we first notice that if $|x|\le 4mR$, the inequality \eqref{diff_2} is  obtained from \eqref{diff_1}. Then  we only need to consider the case $|x|>4mR$.  
		
		Assume that $|x|>4mR$ and $|z|<R$. Since $I_{2s}$ is the fundamental solution for $(-\Delta)^{s}$ and  the function $x\mapsto \frac{A_{2s}}{|x-z|^{d-2s}}\in C^{\infty}(\R^d)\cap \mathscr{L}^{1}_{2s}(\R^d)$ provided $|x|>4mR$ and $|z|<R$, by applying again  \cite{Abatangelo_2}*{Lemma 1.5} we infer that for every $(x,z)\in B^c_{4mR}\times B_R$
		\begin{equation}\label{point_harmonic}
			\frac{A_{2s} C_{d,m,s}}{2}\int_{\R^d}\frac{1}{|y|^{d+2s}}\left(\sum_{k=-m}^{m}(-1)^{k}\binom{2m}{m-k}\frac{1}{|x+ky-z|^{d-2s}}\right)dy=0.
		\end{equation}
		Hence, by combining \eqref{J_1}, \eqref{point_harmonic} with the fact that $\eta_R(x)=1$ if $|x|>4mR$,  we deduce the equality
		\begin{equation}\label{split_3}
			\begin{split}
				&\!\frac{A_{2s} C_{d,m,s}}{2}\! \int_{\R^d}\!|\bar{y}-z|^{-(d-2s)}\! \underbrace{\left(\sum_{k=-m, k\neq 0}^{m}\!\!(-1)^{k}\binom{2m}{m-k} \frac{(\eta_R(\bar{y})-1)}{|x+k\bar{y}|^{(d+2s)}}\right)}_{:=h_R(x,\bar{y})} d\bar{y}\\
				&=\frac{A_{2s} C_{d,m,s}}{2}\int_{\R^d}\frac{1}{|y|^{d+2s}}\left(\,\sum_{k=-m}^{m}(-1)^{k}\binom{2m}{m-k}\frac{(\eta_R(x+ky)-1)}{|x+ky-z|^{d-2s}}\right)dy\\
				&=J_{R}(x,z).
			\end{split}
		\end{equation}
		Furthermore, for fixed $|x|>4mR$, we see that if $|\bar{y}|\ge 2R$ then $h_{R}(x,\bar{y})=0$ while if $|\bar{y}|<2R$ we still have $|x+k\bar{y}|>2mR$. We have therefore proved that, 
		$$h_R(x,\cdot )\in C^{\infty}_c(\R^d)\quad \forall |x|>4mR.$$
		Furthermore, by differentiating, we have the following estimate
		\begin{equation}\label{estimate_4}
			|D^n_{\bar{y}}h_{R}(x,\bar{y})|\le CR^{-|n|}|x|^{-(d+2s)} \quad \forall |x|>4mR,\ n\in \N^d,
		\end{equation}
		for some positive constant $C$. Hence, putting together \eqref{split_3} with \eqref{estimate_4} yields
		\begin{equation}\label{estimate_final}
			|D^{n}_z J_R(x,z)|\le \frac{C}{R^{|n|}|x|^{d+2s}}\int_{B_{2R}}\frac{d\bar{y}}{|\bar{y}-z|^{d-2s}}\le \frac{\tilde{C}}{R^{|n|-2s}|x|^{d+2s}}.
		\end{equation}
		In particular, from \eqref{estimate_final}, if $|x|>4mR$ we deduce \eqref{diff_2}. Finally, since $u\in \mathscr{L}^{1}_{2s}(\R^d)$,  from \eqref{u_repre} and \eqref{diff_2} we can differentiate under the sign of the integral obtaining that for every $n\in \N^{d}$ there exists $C>0$ such that
		\begin{equation*}
			|D^{n} u(z)|\le \frac{C
			}{R^{|n|-2s}}\int_{\R^d}\frac{|u(x)|}{R^{d+2s}+|x|^{d+2s}}dx\quad \forall |z|<R,
		\end{equation*}
		concluding the proof.
	\end{proof}
The following result  refines Lemma \ref{lemma_1} by further assuming some Lebesgue integrability.
	\begin{corollary}\label{l_p}
		Let $k\in \N$, $d\ge 1$, $0<s<\frac{d}{2}$ and $R>0$. Let $u=\sum^{k}_{i=1} u_i$, where $u_i\in L^{p_i}(\R^d)$ and $p_i\in [1,\infty)$. If $u$ solves 
		\begin{equation*}\label{eq_0}
			(-\Delta)^{s}u=0\quad \text{in}\ \mathscr{D}'(B_{2R}),
		\end{equation*}
		then $u\in C^{\infty}({B_R})$ and moreover, for every multi-index $n\in \N^{d}$ we have
		\begin{equation}\label{ineq8boh}
			\|D^{n}u\|_{L^{\infty}({B_R})}\le C\sum^{k}_{i=1}R^{-|n|-\frac{d}{p_i}}\|u\|_{L^{p_i}(\R^d)},
		\end{equation}
		for some positive constant $C$ indenpendent of $R$.
	\end{corollary}
	\begin{proof}
		Clearly $u\in \mathscr{L}^{1}_{2s}(\R^d)$. Hence, Lemma \ref{lemma_1} yields
		\begin{equation}\label{stima_2}
			\|D^{n}u\|_{L^{\infty}(B_R)}\le CR^{2s-|n|}\int_{\R^d}\frac{|u(x)|}{(R^{d+2s}+|x|^{d+2s})}dx.
		\end{equation}
		Then, if we denote by $\omega_{d-1}$ the surface area of the unitary ball in $\R^d$ and $p'_i$ the H\"older conjugate of $p_i$, we obtain 
		\begin{equation}\label{stima_3}
			\begin{split}
				\int_{\R^d}\frac{dx}{(R^{d+2s}+|x|^{d+2s})^{p'_i}} =\omega_{d-1} R^{d-(d+2s)p'_i}\int_{0}^{\infty}\frac{s^{d-1}}{{(1+s^{d+2s})}^{p'_i}}ds.
			\end{split}
		\end{equation}
		Next, by combining \eqref{stima_2}, \eqref{stima_3} and H\"older inequality we derive \eqref{ineq8boh}.
	\end{proof}
	\begin{proof}[Proof of Theorem \ref{poli}]
	
	\noindent
		From Lemma \ref{lemma_1} we obtain that $u$ is smooth and
		$$\|D^{n}u\|_{L^{\infty}({B_R})}\le CR^{2s-|n|}\int_{\R^d}\frac{|u(x)|}{(R^{d+2s}+|x|^{d+2s})}dx\quad 	\forall n\in \N^{d},$$
		where $C$ is a constant not depending on $R$. By sending $R$ to infinity we obtain that $D^n u \equiv 0$ for every $n\in \N^d$ such that $|n|\ge 2s$.
	\end{proof}

 Next, we present a result in the spirit of \cite{mitidieri}*{Theorem 2.4} where the authors find representation formulae in the case of the polyharmonic operator $(-\Delta)^{k},\, k~{\in \N}$, as well as related Liouville theorems.

\begin{proof}[Proof of Theorem \ref{lemmamit}]
\noindent We only focus on $(i)\Rightarrow (ii)$, the other implication being trivial. By \cite{mazya}*{eq. (2.3)}  we have that $U_T=I_{2s}*T$. Moreover, by \eqref{weak_distr}, 
 \begin{align}\label{rz_solve}
 (-\Delta)^{s}(I_{2s}*T)=T\quad \text{in}\ \mathscr{D}'(\R^d).
 \end{align}
Hence, by Theorem \ref{poli} we infer that $u$  can be written as
 \begin{equation*}
 u(x)=(I_{2s}*T)(x)+P(x),\quad P(x)=\sum_{i=0}^{\text{deg(P)}}P_i(x),
 \end{equation*}
 where $P_i$ is a homogeneous polynomial of degree $i$ and by $\text{deg(P)}$ we denote the degree of $P$. 
In particular,  there exists a positive constant $C$ such that
\begin{equation}\label{low_poli}
\begin{split}
\frac{1}{R^d}\int_{R<|x-y|<2R} |P(y)|dy
 \ge CR^{\text{deg(P)}} +o(R^{\text{deg(P)}})\quad \text{as}\ R\to +\infty.
\end{split}
\end{equation}

Assume now by contradiction that $u$ satisfies \eqref{eq2228} and $P\not\equiv l$. Then, by \eqref{low_poli}  we conclude that $\text{deg(P)}\ge 1$ and there exists a positive constant $C$ such that 
\begin{equation}\label{518}
\begin{split}
  \frac{1}{R^d} \int_{R<|y-x|<2R} |(I_{2s}*T)(y)+P(y)-l|dy
  \ge CR^{\text{deg(P)}}+o(R^{\text{deg(P)}}),
\end{split}
\end{equation}
where we have used that $P\not\equiv l$ and the Sobolev embedding to conclude that
$$\lim_{R\to +\infty}\frac{1}{R^d}\int_{R<|y-x|<2R}(I_{2s}*T)(y)dy=0.$$
The inequality \eqref{518} contradicts \eqref{eq2228} completing the proof. 
\end{proof}
\vspace{1mm}

\remark\label{remarknew} We recall that Theorem \ref{lemmamit} admits several variations in the spirit of \cite{newbook}*{Corollary 4.9}. As a matter of fact, by arguing as in the first part of the proof of Theorem \ref{lemmamit}, for any $u\in \mathscr{L}^{1}_{2s}(\R^d)$ and $0\le  T\in \dot{H}^{-s}(\R^d)\cap L^{1}_{loc}(\R^d)$  satisfying  
\begin{equation}\label{distr5}
	(-\Delta)^{s}u=T\quad \text{in}\ \mathscr{D}'(\R^d), 
\end{equation}
we conclude that 
$u=I_{2s}*T+P $
for some polynomial $P$. In particular, by \eqref{nomr_4}
\begin{equation*}
	\|u-P\|_{\dot{H}^{s}(\R^d)}=\|T\|_{\dot{H}^{-s}(\R^d)}.
\end{equation*}
Moreover, in Theorem \ref{lemmamit}, the assumption $0\le T\in \dot H^{-s}(\R^d)\cap L^{1}_{loc}(\R^d)$ can be replaced for example by $T\in L^{q}(\R^d)$ with $1<q<\frac{d}{2s}$, since in this regime the Riesz potential of $T$ is well defined, it belongs to $L^{\frac{dq}{d-2sq}}(\R^d)$, and solves \eqref{rz_solve}; cf. \cite{Stein}*{Chapter V, Theorem 1} or \cite{mazya}*{Lemma 1.8}.
\vspace{1.5mm}

To conclude the section, we formulate another immediate corollary of Theorem {\ref{poli}} concerning the regularity of Riesz potential operator. 
Such outcome, as well as the ones proved in Section~{\ref{sec3}, can be applied for instance to the study of Thomas--Fermi type integral equations, see e.g., \cite{TF} and references therein.}
\begin{corollary}\label{riesz_sobolev} 
	Let $d\ge 1$, $0<s<\tfrac{d}{2}$ and $1< q<\frac{d}{2s}$. Assume that $T\in \dot{H}^{-s}(\R^d)\cap L^{q}(\R^d)$. Then $I_{2s}*T=U_T$, where $U_T$ is the potential defined by \eqref{weak_laplac}. In particular, ${I_{2s}*T}~\in \dot H^{s}(\R^d)\cap \dot H^{2s,q}(\R^d)$.
\end{corollary}
	\begin{proof}
		Since $T\in \dot{H}^{-s}(\R^d)\cap L^{q}(\R^d)$, by \eqref{weak_laplac} and \eqref{weak_distr} there exists an element ${U}_T\in \dot{H}^{s}(\R^d)$ such that 
		\begin{equation}\label{distrib_5}
			(-\Delta)^s U_T=T\quad \text{in}\ \mathscr{D}'(\R^d).
		\end{equation}
		On the other hand, since $1<q<\frac{d}{2s}$, the Riesz potential $I_{2s}*T$  is well defined as a Lebesgue integral and solves \eqref{distrib_5} as well. Note that, by the Sobolev embedding \eqref{hom_embedding} and  \cite{mazya}*{Lemma 1.8},  ${U}_T\in L^{2^*}(\R^d)$ and $I_{2s}*T\in \dot H^{2s,q}(\R^d)\hookrightarrow L^{\frac{dq}{d-2sq}}(\R^d) $. Hence, ${U}_T-I_{2s}*T \in \mathscr{L}^{1}_s(\R^d)$ is $s$--harmonic.
		Then, again in view of \eqref{hom_embedding}, we can  apply Corollary~{\ref{l_p}} to $u:=U_T-I_{2s}*T$ and, by sending $R$ to infinity, we derive the equality
		${U}_T=I_{2s}*T$. 
	\end{proof}
\vspace{-1.5mm}
\remark \label{RZ_ext} Let us fix any $k\in \N$, $1<q<\frac{d}{2s}$ and $1\le p_i<\infty$ for $i=1,\dots, k$.  If we only assume  that $T\in  L^{q}(\R^d)$ and  $u=\sum_{i=1}^{k}u_i$ (with $u_i\in L^{p_i}(\R^d)$) satisfy \eqref{distr5},  by arguing as in the proof of Corollary \ref{riesz_sobolev}, we conclude that $u=I_{2s}*T$.
	
	\section{Br{e}zis--Browder type results}\label{sec4}
In this section we first focus on proving validity of Brezis--Browder type results for fractional Sobolev spaces defined as in \eqref{X}, under extra (local)integrabilty conditions. Next, as a matter of completeness, in the low order regime $0<s<1$, we show that the same holds true without further integrability in the framework of the fractional Sobolev spaces defined via the Gagliardo--Slobodecki\u{\i} seminorm. This provides an improvement to Theorem \ref{integr_alpha_big} for $0<s<1$ and $p=2$; see Theorem \ref{prod_integr} and Remark \ref{rm_last}. 

To begin with, we recall the following technical result:
\begin{proposition}\label{converging_H}
Let $d\ge 1$,  $s>0$, $1<p<\infty$ and $X$ be as in \eqref{X}. Assume that $u\in X$ and $\varphi\in C^{\infty}_c(\bbr^d)$. If we define $\varphi_\lambda(x):=\varphi(\lambda^{-1}x)$ with $\lambda\ge 1$, then $u\varphi_\lambda\in {H}^{s,p}(\bbr^d)$. Moreover, if $\varphi=1$ in a neighborhood of the origin then
\begin{align*}
\|u \varphi_\lambda-u\|_{X}\to 0\quad \text{as}\ \lambda\to \infty.
\end{align*}
\end{proposition}
\begin{proof}
Assume first that $X=\dot H^{s,p}(\R^d)$. Then, from \cite{mult_mazya}*{Chapter 3, Section 3.7} we conclude that $u\varphi_{\lambda}\in \dot H^{s,p}(\R^d)$ and
\begin{align*}
\|u\varphi_\lambda\|_{\dot H^{s,p}(\R^d)}\le C_{d,s,\varphi}\|u\|_{\dot H^{s,p}(\R^d)},
\end{align*}
for some positive constant $C_{d,s,\varphi}$ (independent of $\lambda$).
Moreover, by the Sobolev embedding \eqref{hom_embedding} we further have that $u\varphi_{\lambda}\in H^{s,p}(\R^d)$. Next, by arguing as in \cite{multiplier}*{Lemma 5} we infer that $u\varphi_{\lambda}$ converges to $u$ in $\dot H^{s,p}(\R^d)$. 

Assume now that $X=H^{s,p}(\R^d)$. If $s\in \N$, by performing an explicit computation it is easy to verify that
\begin{align}\label{integer_case}
\|u \varphi_\lambda\|_{H^{s,p}(\R^d)}\le C_{d,s,\varphi}\|u\|_{H^{s,p}(\R^d)}. 
\end{align}
If  $s\notin \N$, then \eqref{integer_case} follows by interpolation with the integer cases. As a result, by arguing again as in \cite{multiplier}*{Lemma 5} we conclude that $u\varphi_{\lambda}$ converges to $u$ in $H^{s,p}(\R^d)$.
\end{proof}
An intermediate step toward Theorem \ref{integr_alpha_big} follows at once. Indeed, let us focus on $s\notin \N$. If $X$ and $\bar{q}$ are as in \eqref{X} and \eqref{q_}, 
and $T\in X{'}\cap L^{\bar{q}}_{loc}(\R^d)$,  by a density argument and the Sobolev embeddings \eqref{hom_embedding}-\eqref{embeddings_in}, it is easy to verify that
\begin{align}
\langle T, u \rangle_{ X{'}, \, X}=\int_{\R^d}T(x)u(x)dx\quad \forall u\in X_c,
\label{lemma_q}
\end{align}
where $X_{c}$ denotes the set of $X$-functions with compact support.  Hence, in view of Proposition \ref{converging_H} and \eqref{lemma_q}, we are now ready to state our main result of this section: 
	\begin{theorem}\label{integr_alpha_big}
Let $d\ge 1$,  $s>0$, $1<p<\infty$, $X$  be as in \eqref{X} and $\bar{q}$ as in \eqref{q_}. Assume that $T\in X{'}\cap L_{loc}^{\bar{q}}(\bbr^d)$ and $u\in X$ is such that $Tu\ge -|f|$  for some  $f\in L^{1}(\bbr^d)$. Then, $Tu\in L^{1}(\bbr^d)$ and 
		$$\langle T,u \rangle_{X^{'}\!,\,X}=\int_{\bbr^d}T(x)u(x)dx.$$
	\end{theorem}
	\begin{proof}
	Assume first that $s\notin \N$. Let us consider  $\varphi\in C^{\infty}_c(\bbr^d)$, $0\le \varphi\le 1$, and $\varphi=1$ in $B_1$. Let us also set $u_\lambda(x):=u(x)\varphi(\lambda^{-1}x)$ with $\lambda\ge 1$. 
		In particular, $u_{\lambda}\in X_c$. Thus, Proposition \ref{converging_H} and \eqref{lemma_q}  yield
		\begin{equation}
		\label{9boh}
			\lim_{\lambda\to \infty}\langle T,u_\lambda\rangle_{X^{'},\, X} =\lim_{\lambda\to \infty}\int_{\bbr^d}T(x)u_\lambda(x)dx= \langle T,u\rangle_{X^{'},\, X}.
		\end{equation}
		Therefore,  by the inequality $Tu\ge -|f|\in L^{1}(\bbr^d)$,  Fatou's Lemma and \eqref{9boh}, we have
		\begin{equation}\label{fatou_1}
			\begin{split}
				\int_{\bbr^d}T(x)u(x)dx
				\le \liminf_{\lambda\to +\infty}\int_{\bbr^d}T(x)u_\lambda(x)dx
				=\langle T,u\rangle_{X^{'},\, X}.
			\end{split}
		\end{equation}
		By \eqref{fatou_1} and using again  the inequality $Tu\ge -|f|\in L^{1}(\bbr^d)$ we also get $Tu\in L^{1}(\bbr^d)$. Finally, since $|Tu_\lambda|\le |Tu|\in L^{1}(\bbr^d)$, by dominated convergence we deduce that
		$Tu_\lambda$ converges to $Tu$ in $L^{1}(\bbr^d)$, concluding the proof in the non-integer setting.
		
Assume that $s\in \N$. If $X=H^{s,p}(\R^d)$ we directly refer to \cite{Adams}*{Theorem 3.4.2}. If $X=\dot H^{s,p}(\R^d)$ and $1<p<\frac{d}{s}$ we simply argue as follows. By the trivial embedding $H^{s,p}(\R^d)\hookrightarrow  \dot H^{s,p}(\R^d)$, $Tu_{\lambda}\ge -|f|$, the fact that $u_\lambda\in H^{s,p}(\R^d)$ and \cite{Adams}*{Theorem 3.4.2}, we infer 
\begin{align}
\langle T,u_\lambda\rangle_{\dot{H}^{-s,p'},\, \dot{H}^{s,p}}=\langle T,u_\lambda\rangle_{{H}^{-s,p'},\, {H}^{s,p}}=\int_{\R^d}T(x)u_{\lambda}(x)dx.
\label{int_eq}
\end{align}
Then, by  \eqref{int_eq} and arguing as above we conclude the proof. 
	\end{proof}

	\subsection*{Low order regime: $0<s<1$.}
	
Up to this point, our analysis has focused on fractional Sobolev spaces characterized via the Riesz and Bessel potentials. By replacing the former spaces with those defined through the Gagliardo--Slobodecki\u{\i}  seminorm,  in Theorem \ref{prod_integr} we show that, for $0<s<1$, an analogue of Theorem \ref{integr_alpha_big} holds in such a different framework, with no further local integrability assumptions required. Namely,
for $d\ge 1$, $0<s<1$ and $1<p<\infty$,  we can define $\dot W^{s,p}(\R^d)$ and $W^{s,p}(\R^d)$  to be the completion of $C^{\infty}_c(\R^d)$ with respect to 
\begin{align*}
[u]_{\dot W^{s,p}(\R^d)}:=\bigg(\int_{\R^d}\int_{\R^d}\frac{|u(x)-u(y)|^p}{|x-y|^{d+sp}}dx\,dy\bigg)^{\frac{1}{p}},
\end{align*}
respectively to
\begin{align}\label{homPbis}
\|u\|_{W^{s,p}(\R^d)}:=\|u\|_{L^p(\R^d)}+[u]_{\dot W^{s,p}(\R^d)}.
\end{align}
In particular, $ W^{s,p}(\R^d)$ (respectively $\dot W^{s,p}(\R^d)$ by further assuming $1<p<\frac{d}{s}$) can be identified as the space of functions $u\in L^{p}(\R^d)$ (respectively $L^{p^*}(\R^d)$) such that $[u]_{\dot W^{s,p}(\R^d)}$ is finite.  
Moreover, similarly to \eqref{X}, we then define $Y$ to be 
\begin{align}\label{X_tilde}
{Y}:= \dot W^{s,p}(\R^d)\ \big(\text{by further assuming}\ 1<p<\tfrac{d}{s}\big)\ \text{or}\ W^{s,p}(\R^d),
	\end{align}
and we understand $Y'\cap L^{1}_{loc}(\R^d)$ in analogy with \eqref{rho_distribution}-\eqref{ext_linear}. 
It is well known that this construction gives rise, in general,  to a space distinct from $X$ in \eqref{X}; see, for instance, \cite{interpolation}*{Theorem 3.26} for a comprehensive treatment of the inhomogeneous case, and \cite{characterisation} together with the references therein for the homogeneous counterpart.
The proof of Theorem \ref{prod_integr} is now based on Proposition \ref{one_sided_approx} where, unlike Proposition \ref{converging_H}, we provide a bounded approximating sequence for elements in $Y$.

In what follows, we define $[u]_{+}:=\max\{u,0\}$ and $[u]_{-}:=|u|-[u]_{+}$.
	\begin{proposition}\label{one_sided_approx}
		Let $d\ge 1$,  $0<s<1$, $1<p<\infty$ and $Y$ be as in \eqref{X_tilde}. Let  $u \in Y$ be a non-negative function.  Then, there exists a sequence $(u_n)_n \subset Y\cap L^{\infty}_{c}(\bbr^d)$ of non-negative functions such that
		\begin{itemize}
			\item[(i)]$0\le u_n(x)\le u(x)$;
			\vspace{1mm}
			\item[(ii)]$u_n\to u$ in $Y$ as $n\to \infty$.
		\end{itemize}
	\end{proposition}
	\begin{proof} If $Y=\dot W^{s,p}(\R^d)$ and $1<p<\frac{d}{s}$,  by \cite{characterisation}*{Lemmas A.1, B.1}  there exists a  non-negative sequence  $(\varphi_n)_n\subset C^{\infty}_{c}(\bbr^d)$ converging to $u$ in $\dot{W}^{s,p}(\bbr^d)$. Similarly, if $Y=W^{s,p}(\R^d)$, thanks to the $L^p$-integrability, \cite{characterisation}*{Lemma B.1} easily adapts to the cases $p\ge \frac{d}{s}$ as well. As a result, for any $0\le u\in Y$ we can find a non-negative sequence  $(\varphi_n)_n\subset C^{\infty}_{c}(\bbr^d)$ converging to $u$ in $Y$. Next, we define \begin{equation*}\label{def_min}
			u_n:=\min\left\{\varphi_n,u\right\}
		=u-{[\varphi_n-u]}_{-}.
		\end{equation*} 
		By construction, $0\le u_n\le u$.
		Moreover, from $\big[[\, \cdot\, ]_{\pm}\big]_{\dot {W}^{s,p}(\bbr^d)}\le [\, \cdot\, ]_{\dot {W}^{s,p}(\bbr^d)}$ (cf.  \cite{Musina}*{page $3$} for the case $p=2$, the case $p\neq 2$ following by the same argument) and \eqref{homPbis},  we conclude that $u_n\in Y\cap L^{\infty}_c(\R^d)$ and  converges to $u$ in $Y$.
	\end{proof}
	\begin{theorem}\label{prod_integr}
		Let $d\ge 1$, $0<s<1$, $1<p<\infty$ and $Y$ be as in \eqref{X_tilde}. Let  $T\in {Y}'\cap L^{1}_{loc}(\bbr^d)$ and  $u\in Y$. If $Tu\ge -|f|$ for some $f\in L^{1}(\R^d)$ then $Tu\in L^{1}(\R^d)$ and
		\begin{equation*}
			\langle T, u\rangle_{Y',\, Y}=\int_{\bbr^d}T(x)u(x)dx.
		\end{equation*}
	\end{theorem}
\begin{proof}
		Let $u\in Y$. Then, $[u]_{\pm}\in Y$ and, by Proposition \ref{one_sided_approx} there exist two sequences ${(u^{\pm}_n)}_{n}\subset Y\cap L^{\infty}_c(\R^d)$ converging to $[u]_{\pm}$ in $Y$ such that $0\le u^{\pm}_n(x)\le [u(x)]_{\pm}$ a.e. Then, the function $\bar{u}_n:=u^{+}_n-u^{-}_{n}$ converges to $u$ in $Y$. Furthermore, by a direct computation we have that 
		$|\bar{u}_n(x)|\le |u(x)|$ and $\bar{u}_n(x)u(x)\ge 0$. Next, by approximating $\bar{u}_n\in Y\cap L^{\infty}_c(\R^d)$ via convolution with mollifiers $(\rho_k)_k$ (see again \cite{characterisation}*{Lemma A.1}), we notice that for every fixed $n$
		\begin{equation*}
			\begin{split}
				\langle T, \bar{u}_n\rangle_{Y',\, Y} 
				&=\lim_{k\to +\infty}\int_{\R^d}T(x)(\bar{u}_n*\rho_k)(x)dx\\
				&= \lim_{k\to +\infty}\int_{\text{supp}(\bar{u}_n)}(T*\rho_k)(x)\bar{u}_n(x)dx =\int_{\R^d}T(x)\bar{u}_n(x)dx.	
			\end{split}
		\end{equation*}
		Then, by arguing as in the proof of Theorem \ref{integr_alpha_big} we infer that $Tu\in L^{1}(\R^d)$ and 
		\begin{equation*}\label{eq10_boh}
			\begin{split}
				\langle T, u\rangle_{Y',\, Y}=\!\lim_{n\to +\infty}\langle T, \bar{u}_n \rangle_{Y',\,Y}\!
				=\! \lim_{n\to +\infty}\int_{\R^d}T(x)\bar{u}_n(x)dx\!=\!\int_{\R^d}T(x)u(x)dx,
			\end{split}
		\end{equation*}
	concluding the proof.
	\end{proof}
\remark \label{rm_last}Because of the equivalence between $Y$ and $X$ if $p=2$ and $0<s<1$ (see e.g., \cite{classicbook}*{Propositions 1.37, 1.59}), Theorem \ref{prod_integr} provides a sharper result than Theorem \ref{integr_alpha_big} in this range of the parameters.
\vspace{1mm}

\noindent
\section*{Acknowledgements.} The author was supported by the European Research Council
(grant no. 864138 ``SingStoch\\
DispDyn''). The author was also funded by the EPSRC Maths DTP
2020 Swansea University, UK (EP/V519996/1). The author would like to thank an anonymous referee for the helpful
comments which have improved the presentation of the paper.

\bibliographystyle{amsplain}

\begin{thebibliography}{99}
\bib{Abatangelo_1}{article}{	
author={Abatangelo, Nicola},
 author={Jarohs, Sven},
  author={Salda\~na, Alberto},
title={Green function and Martin kernel for higher-order fractional
  Laplacians in balls},
   journal={Nonlinear Anal.},
   volume={175},
   date={2018},
   pages={173--190},
   issn={0362-546X},
   review={\MR{3830727}},
   doi={10.1016/j.na.2018.05.019},
   	}			

\bib{Abatangelo_2}{article}{
author={Abatangelo, Nicola},
 author={Jarohs, Sven},
  author={Salda\~na, Alberto},
title={Positive powers of the Laplacian: from hypersingular integrals to boundary value problems},
journal={Comm. Pure Appl. Math.},
			volume={17},
			number = {3},
doi = {10.3934/cpaa.2018045},
			date={2018}
		}
	
	
	
	\bib{newbook}{book}{
		author={N. Abatangelo},
		author={S. Dipierro},
		author={E. Valdinoci}, 
		title={A gentle invitation to the fractional world}, 
		publisher={arXiv:2411.18238},
	}
	
	
	\bib{Adams}{book}{
 author={Adams, David R.},
   author={Hedberg, Lars Inge},
   title={Function spaces and potential theory},
   series={Grundlehren der mathematischen Wissenschaften [Fundamental
   Principles of Mathematical Sciences]},
   volume={314},
   publisher={Springer-Verlag, Berlin},
   date={1996},
   pages={xii+366},
   isbn={3-540-57060-8},
   review={\MR{1411441}},
   doi={10.1007/978-3-662-03282-4},}
		
		
		\bib{classicbook}{book}{
			  author={Bahouri, Hajer},
   author={Chemin, Jean-Yves},
   author={Danchin, Rapha\"el},
   title={Fourier analysis and nonlinear partial differential equations},
   series={Grundlehren der mathematischen Wissenschaften [Fundamental
   Principles of Mathematical Sciences]},
   volume={343},
   publisher={Springer, Heidelberg},
   date={2011},
   pages={xvi+523},
   isbn={978-3-642-16829-1},
   review={\MR{2768550}},
   doi={10.1007/978-3-642-16830-7},
		}
		
		\bib{interpolation}{article}{
		author={J. C. Bellido},
		author={G. G.-Sáez}, 
		title={Bessel Potential Spaces and Complex Interpolation: Continuous embeddings}
		journal={arXiv:2503.04310}}
		
		\bib{1990}{article}{
			 author={Boccardo, L.},
   author={Giachetti, D.},
   author={Murat, F.},
   title={A generalization of a theorem of H. Br\'ezis \& F. E. Browder and
   applications to some unilateral problems},
   language={English, with French and Italian summaries},
   journal={Ann. Inst. H. Poincar\'e{} C Anal. Non Lin\'eaire},
   volume={7},
   date={1990},
   number={4},
   pages={367--384},
   issn={0294-1449},
   review={\MR{1067781}},
   doi={10.1016/S0294-1449(16)30297-9},
		}

\bib{BBK}{article}{	
 author={Bogdan, Krzysztof},
   author={Byczkowski, Tomasz},
   title={Potential theory for the $\alpha$-stable Schr\"odinger operator on
   bounded Lipschitz domains},
   journal={Studia Math.},
   volume={133},
   date={1999},
   number={1},
   pages={53--92},
   issn={0039-3223},
   review={\MR{1671973}},
   doi={10.4064/sm-133-1-53-92},	
}		

\bib{BBK2}{article}{
  author={Bogdan, K.},
   author={Kulczycki, T.},
   author={Nowak, Adam},
   title={Gradient estimates for harmonic and $q$-harmonic functions of
   symmetric stable processes},
   journal={Illinois J. Math.},
   volume={46},
   date={2002},
   number={2},
   pages={541--556},
   issn={0019-2082},
   review={\MR{1936936}},
  }
		
		\bib{characterisation}{article}{
			 author={Brasco, Lorenzo},
   author={G\'omez-Castro, David},
   author={V\'azquez, Juan Luis},
   title={Characterisation of homogeneous fractional Sobolev spaces},
   journal={Calc. Var. Partial Differential Equations},
   volume={60},
   date={2021},
   number={2},
   pages={Paper No. 60, 40},
   issn={0944-2669},
   review={\MR{4225499}},
   doi={10.1007/s00526-021-01934-6},
		}
		
		\bib{cartan}{article}{
			   author={Br\'ezis, Ha\"im},
   author={Browder, Felix},
   title={A property of Sobolev spaces},
   journal={Comm. Partial Differential Equations},
   volume={4},
   date={1979},
   number={9},
   pages={1077--1083},
   issn={0360-5302},
   review={\MR{0542513}},
   doi={10.1080/03605307908820120},
		}
		
		\bib{1982}{article}{
 author={Br\'ezis, Ha\"im},
   author={Browder, Felix E.},
   title={Some properties of higher order Sobolev spaces},
   journal={J. Math. Pures Appl. (9)},
   volume={61},
   date={1982},
   number={3},
   pages={245--259 (1983)},
   issn={0021-7824},
   review={\MR{0690395}},
		}	

\bib{mitidieri}{article}{
 author={Caristi, Gabriella},
   author={D'Ambrosio, Lorenzo},
   author={Mitidieri, Enzo},
   title={Representation formulae for solutions to some classes of higher
   order systems and related Liouville theorems},
   journal={Milan J. Math.},
   volume={76},
   date={2008},
   pages={27--67},
   issn={1424-9286},
   review={\MR{2465985}},
   doi={10.1007/s00032-008-0090-3},
}


  
		
			
			\bib{1984}{article}{
				 author={Coron, Jean-Michel},
   title={On a problem of H. Brezis and F. Browder concerning Sobolev
   spaces},
   journal={Indiana Univ. Math. J.},
   volume={33},
   date={1984},
   number={2},
   pages={179--183},
   issn={0022-2518},
   review={\MR{0733895}},
   doi={10.1512/iumj.1984.33.33009},
			}
			
			
			\bib{DuPlessis}{book}{
				 author={du Plessis, Nicolaas},
   title={An introduction to potential theory},
   series={University Mathematical Monographs},
   volume={No. 7},
   publisher={Hafner Publishing Co., Darien, CT; Oliver and Boyd, Edinburgh},
   date={1970},
   pages={viii+177},
   review={\MR{0435422}},
			}
			
			\bib{entire}{article}{
				   author={Fall, Mouhamed Moustapha},
   title={Entire $s$-harmonic functions are affine},
   journal={Proc. Amer. Math. Soc.},
   volume={144},
   date={2016},
   number={6},
   pages={2587--2592},
   issn={0002-9939},
   review={\MR{3477075}},
   doi={10.1090/proc/13021},
			}
			
			\bib{garofalo}{article}{
				  author={Garofalo, Nicola},
   title={Fractional thoughts},
   conference={
      title={New developments in the analysis of nonlocal operators},
   },
   book={
      series={Contemp. Math.},
      volume={723},
      publisher={Amer. Math. Soc., [Providence], RI},
   },
   isbn={978-1-4704-4110-4},
   date={[2019] \copyright 2019},
   pages={1--135},
   review={\MR{3916700}},
   doi={10.1090/conm/723/14569},
			}
			
			
			
			
			
			\bib{TF}{article}{ 
				 author={Greco, Damiano},
   title={Optimal decay and regularity for a Thomas-Fermi type variational
   problem},
   journal={Nonlinear Anal.},
   volume={251},
   date={2025},
   pages={Paper No. 113698, 37},
   issn={0362-546X},
   review={\MR{4822623}},
   doi={10.1016/j.na.2024.113698},
			}
			
			
			\bib{Hedbger}{article}{
				 author={Hedberg, Lars Inge},
   title={Two approximation problems in function spaces},
   journal={Ark. Mat.},
   volume={16},
   date={1978},
   number={1},
   pages={51--81},
   issn={0004-2080},
   review={\MR{0499137}},
   doi={10.1007/BF02385982},
			}
			
			
			
			
			\bib{mazya}{article}{
				author={V.P. Khavin},
				author={V.G. Maz'ya},
				title={Non-linear potential theory},
				date={1972},
				journal={Russ. Math. Surv.},
				volume={27},
				issue={6},
				number={71},
				doi={10.1070/RM1972v027n06ABEH001393}
			}
			
			
			
			
\bib{orbital}{article}{
author={Lu, Jianfeng},
   author={Moroz, Vitaly},
   author={Muratov, Cyrill B.},
   title={Orbital-free density functional theory of out-of-plane charge
   screening in graphene},
   journal={J. Nonlinear Sci.},
   volume={25},
   date={2015},
   number={6},
   pages={1391--1430},
   issn={0938-8974},
   review={\MR{3415051}},
   doi={10.1007/s00332-015-9259-4},
			}
			
			
			
			\bib{mult_mazya}{book}{
			  author={Maz'ya, Vladimir G.},
   author={Shaposhnikova, Tatyana O.},
   title={Theory of Sobolev multipliers},
   series={Grundlehren der mathematischen Wissenschaften [Fundamental
   Principles of Mathematical Sciences]},
   volume={337},
   note={With applications to differential and integral operators},
   publisher={Springer-Verlag, Berlin},
   date={2009},
   pages={xiv+609},
   isbn={978-3-540-69490-8},
   review={\MR{2457601}},
				}
			
			\bib{realization}{article}{
			  author={Monguzzi, Alessandro},
   author={Peloso, Marco M.},
   author={Salvatori, Maura},
   title={Fractional Laplacian, homogeneous Sobolev spaces and their
   realizations},
   journal={Ann. Mat. Pura Appl. (4)},
   volume={199},
   date={2020},
   number={6},
   pages={2243--2261},
   issn={0373-3114},
   review={\MR{4165680}},
   doi={10.1007/s10231-020-00966-7},
			}
			
			
			
			\bib{Musina}{article}{
				 author={Musina, Roberta},
   author={Nazarov, Alexander I.},
   title={A note on truncations in fractional Sobolev spaces},
   journal={Bull. Math. Sci.},
   volume={9},
   date={2019},
   number={1},
   pages={1950001, 7},
   issn={1664-3607},
   review={\MR{3949429}},
   doi={10.1142/S1664360719500012},
   			}
			
			
			\bib{multiplier}{article}{
				 author={Palatucci, Giampiero},
   author={Pisante, Adriano},
   title={Improved Sobolev embeddings, profile decomposition, and
   concentration-compactness for fractional Sobolev spaces},
   journal={Calc. Var. Partial Differential Equations},
   volume={50},
   date={2014},
   number={3-4},
   pages={799--829},
   issn={0944-2669},
   review={\MR{3216834}},
   doi={10.1007/s00526-013-0656-y},
			}
			
			
		
		\bib{Samko}{book}{
  author={Samko, Stefan G.},
   author={Kilbas, Anatoly A.},
   author={Marichev, Oleg I.},
   title={Fractional integrals and derivatives},
   note={Theory and applications;
   Edited and with a foreword by S. M. Nikol\cprime ski\u i;
   Translated from the 1987 Russian original;
   Revised by the authors},
   publisher={Gordon and Breach Science Publishers, Yverdon},
   date={1993},
   pages={xxxvi+976},
   isbn={2-88124-864-0},
   review={\MR{1347689}},
}
			
			
			
			\bib{Stein}{book}{
				 author={Stein, Elias M.},
   title={Singular integrals and differentiability properties of functions},
   series={Princeton Mathematical Series},
   volume={No. 30},
   publisher={Princeton University Press, Princeton, NJ},
   date={1970},
   pages={xiv+290},
   review={\MR{0290095}},
			}
			
	\bib{hypo}{article}{		
	author={Weyl, Hermann},
   title={The method of orthogonal projection in potential theory},
   journal={Duke Math. J.},
   volume={7},
   date={1940},
   pages={411--444},
   issn={0012-7094},
   review={\MR{0003331}},
	}		
\end{thebibliography}

\end{document}